\documentclass{amsart}
\usepackage{amscd,amssymb,amsmath,amsthm}
\usepackage[dvips]{graphicx}
\usepackage[all]{xy}
\theoremstyle{plain}

\newtheorem{proposition}{Proposition}
\newtheorem{theorem}[proposition]{Theorem}

\newtheorem{lemma}[proposition]{Lemma}
\newtheorem{remark}[proposition]{Remark}
\newtheorem*{proposition*}{Proposition}
\newtheorem*{theorem*}{Theorem}
\newtheorem*{corollary*}{Corollary}
\newtheorem*{lemma*}{Lemma}
\newtheorem*{remark*}{Remark}
\newtheorem*{example*}{Example}

\newcommand{\Z}{\mathbb{Z}}
\newcommand{\Q}{\mathbb{Q}}
\newcommand{\R}{\mathbb{R}}

\begin{document}

\title{Quantum Master Equation and Open Gromov-Witten Theory}

\author{Vito Iacovino}


\email{vito.iacovino@gmail.com}

\date{version: \today}


\begin{abstract}
We construct the higher genus Open-Closed Gromov-Witten potential as a solution of the quantum master equation defined up to quantum master isotopy. 

\end{abstract}

\maketitle

\section{Introduction}

The develop of the theory of closed Gromov-Witten invariants has a long history and goes back to three decades ago (see e.g \cite{FO} and reference therein). In contrast,  the existence of an Open Gromov-Witten theory remained  a mystery for long time. The main reason is quite geometrical and stem on the fact that the moduli space of open (pseudo-)holomorphic curves has a not empty boundary. This  makes the count of open pseudoholomophic curves depending drastically on the perturbation of the moduli space.  

On the another hand, in the begin of nineties, Witten  in \cite{W} introduced a physical definition of Open Topological String for a pair $(X,L)$, where $X$ a Calabi-Yau three fold and $L$ is a Maslov index zero Lagrangian submanifold. Witten proposed a relation between the Open Topological String of $(X,L)$ and pertubative Chern-Simons Theory of $L$.

Open Topological String theory should be seen as the physical counterpart of Open Gromov-Witten theory. However the existence of any relation between \cite{W} and the problem of the boundary of the moduli space of pseudo holomorphic curves   went unnoticed  for long time both in the mathematical and the physical community. For long time, the article of Witten has been seen by the mathematical community as stating just a  duality between Lagrangian Floer Homology of the zero section of the cotangent bundle $T^*L$ and pertubative Chern-Simons theory of $L$, in the particular case $X=T^*L$. Mathematicians thought that it was not possible to define Open Gromov-Witten invariants in the same context  Open Topological String was  defined by physicists. 

The existence of a relation between the claim of \cite{W} and the solution of the problem of the boundary of Open Gromov-Witten invariants was first observed in \cite{OGW1}, \cite{OGW2}, where the mathematical solution was made precise.
In \cite{OGW1} we introduced the moduli space of pseudo-holomorphic multi-curves.  Roughly the moduli space of multi-curves consists in the moduli space of connected or unconnected curves whose perturbation is constrained in a suitable way. Multi-curves can be seen as the analogous of multi-instantons familiar in Donalson-Witten theory. 
Using these moduli spaces, in \cite{OGW1},  \cite{OGW2}, we proposed the general solution of the problem of the boundary of the moduli for open Gromov-Witten invariants and we developed a general  mathematical theory of Open Gromov-Witten invariants for the pair $(X,L)$, where $X$ is a Calabi-Yau symplectc six manifold and $L$ is a Maslov index zero Lagrangian submanifold . The moduli space of multi-curves are the natural mathematical counterpart of the pertubative expansion of the Wilson Loops appearing in \cite{W}.

Later, as an elaboration of the notion of system of chains introduced in \cite{OGW1}, in \cite{OGW3}  
we introduced the multi-curve-chain complex.  The relations defining $MC$-cycles correspond to the constrains defining (the perturbation of) the moduli space of multi-curves. Open Gromov-Witten invariants should be understood with values on $MC$-cycles.




In turn out that  $MC$-cycles  are strictly related to the mathematical theory of point splitting  Chern-Simons,  where they play the role of the configuration space of points.
However, we stress that the intuitive picture of \cite{W} can be misleading. 
The relations that Wilson loops satisfy (Skein relations) are not an elementary combination of the fundamental relations defining $MC$-cycles, hence they are not  directly related to the moduli of (multi-)curves. 
Instead, the Wilson loops of \cite{W} are in some sense a coherent combination of elements of $MC$-cycles, objects that we called coherent cycles in \cite{Boundary-States}.



In \cite{OGW3}, for any class $\beta \in H_2(X,L)$,  from the moduli space of multi-curves in the homology class $\beta$ we constructed a $MC$-cycle $Z_{\beta} $. $Z_{\beta}$ can be expressed as a combination of coherent $MC$-cycles (Wilson loops).  The frame of the Wilson loops is related to the  perturbation of the moduli space of constant maps.   It is proved that $Z_{\beta}$, up to isotopy of $MC$-cycles, does not depend on the choices we made in the construction (such as the compatible almost complex structure, pertubations, etc.)


Coupling $Z$ with the abelian Chern-Simons propagator (compare with \cite{OGW2}),
in this paper we construct the Open-Closed Gromov-Witten partition $\mathfrak{P}(Z)$ function.
In principle we would like to define also the Open-Closed Gromov-Witten potential $\mathfrak{W}(Z)$ as the contribution of connected graphs and obtain the relation $\mathfrak{P}(Z)= \text{exp}(\frac{1}{g_s}\mathfrak{W}(Z))$. 
However there is a technicality. The $MC$-cycles $Z_{\beta}$ of \cite{OGW3} are constructed degree by degree, without imposing any relations between $Z_{\beta}$ and $Z_{\beta'}$ with $\omega(\beta')< \omega(\beta)$. In particular, it is not automatically true that $(Z_{\beta})_{\beta}$ can be expressed as the exponential of connected graphs. In order for this to be true $Z$ needs to satisfies what we call factorization property. We adapt the construction of \cite{OGW3} in order to obtain a collection of $MC$-cycles  $(Z_{\beta})_{\beta}$ that satisfies the factorization property.

\section{Multi Curve Cycles}

In this section we review the result of \cite{OGW3}. Fix a compact oriented three manifold $M$.

Fix an oriented three manifold $M$ and an homological  class $\gamma \in H_1(M, \Z)$. 
Consider the  objects  
\begin{equation} \label{generator-nice}
( H , ( w_h  )_{h \in H} ) 
\end{equation}
where $H$ is a finite set , $\{ w_h  \}_{h \in H} $ are closed one  dimensional  chains on $M$ on the homology class $\gamma$  which are close on the $C^0$-topology. 
Denote by $\mathfrak{Gen}(\gamma)^{\dagger}$ the set of the objects (\ref{generator-nice}) modulo the obvious equivalence relation:  $( H , ( w_h  )_{h \in H} ) \cong ( H' , ( w_h'  )_{h \in H'} )$ if there exists an identification of sets $H =H'$ such that $w_h=w_h'$. 

Define $\mathfrak{Gen}(\gamma) \subset \mathfrak{Gen}(\gamma)^{\dagger}$ as the subset obtained imposing the extra condition  
\begin{equation} \label{generator-transversality}
w_h \cap w_{h'} = \emptyset \text{    if    } h \neq h'.
\end{equation}
The vector space $\mathcal{Z}_{\gamma}$ of $\mathbf{nice \thickspace MC-cycles}$ in the homology class $\gamma$  is the formal vector space generated by $\mathfrak{Gen}(\gamma)$.

In this paper we consider $\mathcal{Z}_{\gamma}[[g_s]]$ the formal power series on the formal variable $g_s$ with coefficients  $\mathcal{Z}_{\gamma}$. 




A generator of isotopies of $MC$-cycles consists in an array 
\begin{equation} \label{generator-nice-iso}
 ( H , \{ \tilde{w}_h  \}_{h \in H} , [a,b] ) 
\end{equation}
Hence  an isotopy of $MC$-cycles is defined by a formal power series 
$$\tilde{Z}= \sum_i  g_s^{k_i}   r_i ( H_i , ( \tilde{w}_{h,i}  )_{h \in H_i} , [a_i,b_i] ) ,$$
for some $r_i \in \Q$, $k_i \rightarrow \infty$.
In particular $\tilde{Z}$ defines a one parameter family  of elements of $Z_t \in \mathcal{Z}_{\gamma}$, $t \in [0,1]$, which can be discontinues on a finite number of times (we do not define $Z_t$ if $t$ is a discontinuity point). 
if for $h,h' \in H_i$ , $ \tilde{w}_{h_1,i},  \tilde{w}_{h_2,i}$ cross transversely at a time $t_0$, we require that $\tilde{Z}$ jumps according to the formula
\begin{equation} \label{jump}
 Z^{t_0^+} - Z^{t_0^-} =  \pm g_s^{k_i+1} r_i   (H_i , ( w_{h,i}^{t_0}  )_{h \in H_i \setminus \{h_1,h_2 \} }  )  .
\end{equation}
where the sign is defined by the sign of the crossing.
Denote by  $\tilde{\mathcal{Z}}_{\gamma}$ the set of isotopies of nice $MC$-cycles.

We need also to introduce the space of $MC$-cycles $\mathcal{Z}_{\gamma| \mathfrak{c}}[[g_s]]$ with not vanishing Chern Class $\mathfrak{c}$. This is done as in \cite{OGW3}. 

\subsubsection{Open Gromov-Witten $MC$-cycle}

Let $(X,L)$ be a pair consisting in a Calabi-Yau simplectic six-manifold $X$ and a Maslov index zero lagrangian submanifold $L$. 
We assume $[L] =0 \in H_3(X, \Z)$ and we  fix  a four chain $K$ with $\partial K =L$.

To the four chain $K$ it is associated a Chern Class $\mathfrak{c}(K) \in H_1(L,\Z)$.

\begin{theorem} \label{main-theorem} (\cite{OGW3})
Let $\beta \in H_2(X,L, \Z)$. To the moduli space of pseudoholomotphic multi-curves of homology class $\beta(L)$  it is associated a multi-curve cycle $Z_{\beta} \in \mathcal{Z}_{\partial \beta| \mathfrak{c}} [[g_s]]$ with Chern class $\mathfrak{c}$. $Z_{\beta}$ depends by the varies choices we made to define the  Kuranishi structure and its perturbation on the moduli space of multi-curves.
Different choices lead to isotopic $MC$-cycles.
\end{theorem}



\subsection{MC-cycles including the degrees} \label{nice-MC2-section}

In \cite{OGW3} we also considered nice $MC$-cycles which keep trace of the degrees of the multi-curves. A component of the multi-curve corresponds to an element of $V$ in (\ref{generator2}) below. Here we consider a slightly different version, which forget about the Euler Charactertic of each single component and as above we consider the full $MC$-cycle with arbitrary total Euler Characteristic.

Below we consider nice $MC$-cycles which arise from moduli space of multi-curves associated to graphs without closed components.
(A closed component of a decorated graph $G \in \mathfrak{G}$ considered in \cite{OGW3} is a component $c \in Comp(G)$ with $V_c = D_c = \emptyset$.) That is,  we have discarded  the purely closed contributions from the $MC$-cycle. 

Denote $\Gamma = H_2(X,L)$.
A generator of $MC$-cycle is defined by an array
\begin{equation} \label{generator2}
(V, H,  ( w_{v,h} )_{v \in V,h \in H} ,  ( w^{ann}_h )_{h \in H}, (  \beta_v )_{v \in V}  )
\end{equation}
where $V$ and $H$ are finite sets, $\beta_v \in \Gamma$,   $ \{ w_{v,h} \}_{h \in H}$ are closed one dimensional chain on $M$ on the homology class $\partial \beta_v$ , which are close in the $C^0$-topology.  
 $ \{ w^{ann}_h \}_{h \in H}$  are closed one dimensional chains on $M$ in the homology class $\mathfrak{c}$  which are close in the $C^0$-topology.  

We assume that 
\begin{equation} \label{support-property-nice-cycles}
    \beta_v \notin tors(\Gamma), \quad \parallel \beta_v \parallel \leq C^{supp}  \omega(\beta_v) \text{   for each  } v \in V(G)
\end{equation}

Denote by $\mathfrak{Gen}(\beta)^{\dagger}$ the set of objects (\ref{generator2}) such that $\sum \beta_v = \beta$, modulo the obvious equivalence relation.   

Observe that (\ref{support-property-nice-cycles}) implies that there exists  $N_{\beta } \in \Z_{> 0}$ such that
\begin{equation} \label{bound-vertices}
    |V| \leq N_{\beta } 
\end{equation}
for each generator (\ref{generator2}).

Set $w_h= \sum w_{v,h} + w^{ann0}_h$.
Define $\mathfrak{Gen}(\beta) $ as the subset of $ \mathfrak{Gen}(\beta)^{\dagger}$ obtained imposing the extra condition  
\begin{equation} \label{generator-transversality2}
w_{h} \cap w_{h'} = \emptyset \text{    if    } h \neq h'.
\end{equation}

Let $\mathcal{Z}_{\beta} $ be the set of formal power series
$$ \sum_i g_s^{k_i} Gen_i$$
with $Gen_i \in \mathfrak{Gen}(\beta) $, $k_i \rightarrow \infty$, $k_i + |V_i| \geq 0$. 
The last condition is related to the fact we have discarded  the purely closed contributions from the $MC$-cycle. 

To define  isotopies of nice-$MC$ cycles $\tilde{Z} $  we consider objects 
\begin{equation} \label{generator2-iso}
(V, H,  ( \tilde{w}_{v,h} )_{v \in V,h \in H} ,  ( \tilde{w}_{h}^{ann} )_{h \in H}   , (  \beta_v )_{v \in V} , [a,b]   )
\end{equation}
where $[a,b] \subset [0,1]$,  $ \tilde{w}_{h,v} : F_v \times [a,b] \rightarrow M$ for some one dimensional compact manifold $F_v$. 

Let $\tilde{\mathfrak{Gen}}(\beta)$ be the set of objects (\ref{generator2-iso}) modulo the obvious equivalence relation.

An isotopy of of nice-$MC$ cycles $\tilde{Z} $ is defined by formal power series 
$$\tilde{Z}= \sum_i r_i g_s^{k_i} (V_i, H_i,  ( \tilde{w}_{v,h,i} )_{v \in V_i,h \in H_i} ,  ( \tilde{w}_{h,i}^{ann} )_{h \in H_i}   , (  \beta_v )_{v \in V_i} , [a_i,b_i]   ) ,$$
with $r_i \in \Q$, $k_i \rightarrow \infty$, and $k_i + |V_i| \geq 0$. 

The formal power series are considered modulo gluing of objects (\ref{generator2-iso}).

$\tilde{Z}$ defines a one parameter family of nice $MC$-cycles $(Z^t)_t$ $Z^t \in \mathcal{Z}_{\beta}$ discontinuous for a finite number of times. The discontinuity at the time $t_0$ is obtained by the formula:
\begin{equation} \label{jump2}
 Z^{t_0^+} - Z^{t_0^-} =  \pm r' g_s^{k'} (V', H',  ( w_{v,h}' )_{v \in V',h \in H'} ,( {w_{h}^{ann0}}' )_{h \in H'}, ( \beta_v' )_{v \in V'}  )
\end{equation}
where
\begin{itemize}
    \item If  $\tilde{w}_{v_1, h_1,i}$ crosses $\tilde{w}_{v_2, h_2,i}$  transversely at the time $t_0$ we have two cases:
    \begin{itemize}
        \item $v_1 \neq v_2$: $H'= H_i \setminus \{ h_1,h_2 \} $  , $V' = V_i \setminus \{v_1, v_2 \}  \sqcup v_0$, where $v_0$ is a new vertex with associated data $\beta_{v_0}'= \beta_{v_1} + \beta_{v_2}$,  $ w_{ v_0,h}' = w_{ v_1,h,i}^{t_0^- } +  w_{ v_2,h,i}^{t_0^- } $ for each $h \neq h_1, h_2$. All the other data remain  the same.
\item  $v_1=v_2$:  $H'= H_i \setminus \{ h_1,h_2\} $, $V'=V_i$  .  All the other data remain the same.
    \end{itemize}
    \item If $\tilde{w}_{v, h_1,i}$ crosses $\tilde{w}^{ann}_{ h_2.i}$, $H'= H_i \setminus \{ h_1,h_2\} $, $V'=V_i$  .  All the other data remain the same.
\end{itemize}
$r' = r_i$, $k' = k_i+1$ and the sign is defined by the sign of the crossing.

\section{Open Gromov-Witten Partition Function}

\subsection{Quantum Master Equation}

In this section we recall some basic definition related to the (finite dimensional) Batalin-Vilkovski formalism. 

Fix a super vector space $H$ with an odd symplectic form. Denote by $\mathcal{O}(H)$ the algebra of polynomial functions on $H$.

Let $x_i$, $y_i$ be Darboux coordinates for $H$ with $x_i$ even and $y_i$ odd. Let $\Delta$ be the order two differential operator on $H$ defined by  
$$ \Delta = \partial_{x_i} \partial_{y_i} . $$
The operator $\Delta $ is independent of the choice of basis of $H$.

The bracket on the algebra $\mathcal{O}(H)$ is defined by
$$ \{  f,g  \} = \Delta (fg) - \Delta (f ) g - (-1)^{|f|} f \Delta (g). $$


Denote by $\mathcal{O}(H)[\hbar]$ the polynomial functions with coefficients in the formal parameter $\hbar$. An even element $S \in \mathcal{O}(H)[\hbar]$ satisfies the quantum master equation if 
$$  \Delta e^{W/\hbar}=0 . $$
This equation can be rewritten as
\begin{equation} \label{masterequation}
\frac{1}{2} \{ W, W \} + \hbar \Delta W =0 .
\end{equation}

We will need to consider also the one parameter family version of the above construction.
Consider the space $\Omega^*([0,1])  \otimes \mathcal{O}(H) [\hbar]$. Extend the operator $\Delta$ to this space acting trivially on $\Omega^*([0,1]) $.
A master homotopy is an even element $ \tilde W \in  \Omega^*([0,1]) \otimes  \mathcal{O}(H) [\hbar] $ such that 
\begin{equation} \label{masterhomotopy}
d \tilde W + \frac{1}{2} \{ \tilde W,\tilde W \} + \hbar \Delta \tilde W =0 . 
\end{equation}
Write  $ \tilde W = W(t) + K(t) dt$ where $W(t)$ and $K(t)$ are elements of $\mathcal{O}(H) [\hbar]$. Equation (\ref{masterhomotopy}) becomes
$$ \frac{1}{2} \{ W(t), W(t) \} + \hbar \Delta W(t) =0     $$
$$  \dot W(t) +  \{ K(t), W(t) \} + \hbar \Delta K(t) =0   .  $$

In the case we are interested in $H$ is the cohomology of a compact oriented three manifold $M$
$$H = H^*(M)[1] .$$ 
The odd symplectic form is induced by the pairing
$$\langle \alpha,  \alpha'  \rangle = (-1)^{|\alpha|} \int_M \alpha \wedge \alpha'  . $$

The formal variable $\hbar $ will be identified with the the formal variable $g_s$ used in Gromov-Witten theory.




\subsection{Partition Function}


Given a finite set $H$, denote by $\mathcal{E}(H)$ the set of partitions  of $H$ in sets of cardinality one or two. For $E \in \mathcal{E}(H)$ denote by $E^{ex} \subset E$ the sets of cardinality one, and $E^{in} \subset E$ the sets of cardinality two.




To define the partition function of a $MC$-cycle, by linearity,  it is enough to define the contribution of a generator $ (H, ( w_h )_{h \in H}) \in \mathfrak{Gen}(\gamma)$.
\begin{equation} \label{partiction1}
\mathfrak{P}((H, \{ w_h \}_{h \in H})) =   \sum_{E \in \mathcal{E}(H)} \frac{g_s^{|E_{in}|}}{|\text{Aut}(E)|} \prod_{( h,h' ) \in E^{in}} \langle P, w_h \times w_{h'} \rangle  \times \prod_{ ( h ) \in E^{ex}} \langle \sum_i x_i \alpha_i, w_h \rangle .
\end{equation} 
In the above formula we consider  $w_h \times w_{h'}$ as a two-chain on $Conf_2(M)$ and we use the usual pairing between differential forms and chains.


Formula (\ref{partiction1}) can be extended  to nice $MC$-cycles including the degrees (\ref{generator2}) as follows.
In this case the partition function is written as an expansion on decorated graphs  
 $$G =(V(G), ( H_v )_{v \in V(G)}, E(G), (\beta_v )_{v \in V(G)} , V^{ann0})$$ 
 where  $V^{ann0} \subset V(G)$ is a special set of vertices and the other data are defined as usual. 
 We assume that for each $v \in V^{ann0} $ we have $\beta_v=0$ and $|H_v|=1$ and that there are no other area zero vertices besides $ V^{ann0}$.   Assume also that there exists an identification $V(G) \setminus V^{ann0}= V$ compatible with the data $(\beta_v)_v$.
For each $h$, denote with $v_h$ the vertex with $h \in H_{v_h}$. Set 
$w_{h,v_h} = w_h^{ann} $ if $v_h \in V^{ann0}$,
The contribution to the partition function associated to $G$ is obtained  by the formula
\begin{multline} \label{contribution-graphs}
    \mathfrak{P}( G, (V, H,  ( w_{v,h} )_{v \in V,h \in H} ,  ( w^{ann}_h )_{h \in H}, (  \beta_v )_{v \in V}  )) =  \\
    \sum_{s}   \frac{g_s^{|E^{in}(G)| } } {|\text{Aut}(G)|} \prod_{\{ h,h' \} \in E^{in}} \langle P,w_{s(h),v_h}  \times w_{s(h'),v_{h'}}  \rangle   \times \prod_{ \{ h \} \in E^{ex}} \langle \sum_i x_i \alpha_i, w_{s(h),v_h} \rangle.
\end{multline}
where the sum is made  on the set of bijective maps $s : \sqcup_v H_v  \rightarrow H$.

From (\ref{bound-vertices}) we obtain
$$ \mathfrak{P}(Z) \in  \frac{1}{g_s^{N_{\beta}}} \R[[g_s,x]] $$
if $Z \in \mathcal{Z}_{\beta}$.

\subsection{Isotopies} \label{isotopies-section}

We now extend (\ref{partiction1}) to isotopies of $MC$-cycles, showing that isotopic $MC$-cycles leads to isotopic solutions of the quantum master equation. At the same time we show that different choice of the propagator leads to isotopic solutions of the quantum master equation.



Let $ ( H , \{ \tilde{w}_h  \}_{h \in H} , [a,b] ) $  be a generator of $MC$-isotopies  as in (\ref{generator-nice-iso}) . Let $\tilde{P}$ an isotopy of propagators as in Lemma \ref{parameter-propagator-lemma}.
Define the one form 
$$\mathfrak{P}(( H , \{ \tilde{w}_h  \}_{h \in H} , [a,b] )) \in \Omega^1([a,b])  \otimes \R[x,y] $$  
as
\begin{equation} \label{partiction1-isotopy}
\mathfrak{P}((H, \{ \tilde{w}_h \}_{h \in H}, [a,b])) =   \sum_{E \in \mathcal{E}(H)} \frac{g_s^{|E_{in}|}}{|\text{Aut}(E)|} \prod_{\{ h,h' \} \in E^{in}} \langle \tilde{P}, \tilde{w}_h \times \tilde{w}_{h'} \rangle  \times \prod_{ \{ h \} \in E^{ex}} \langle \sum_i x_i \alpha_i, \tilde{w}_h \rangle .
\end{equation} 
In the formula above we consider $ \tilde{w}_h \times \tilde{w}_{h'}$ as a two chain on $\text{Conf}_2(M)  \times [0,1]$. The pairing appearing in the formula gives a differential form on $[0,1]$, defined using the push-forward of differential forms.


We observe that  (\ref{partiction1-isotopy}) is not continuous but it jumps when a pair $(\tilde{w}_h ,\tilde{w}_{h'})$ crosses each other. This jump is compensate by the jump (\ref{jump}) appearing in the definition of $MC$-isotopy:
\begin{proposition} Let $\tilde{Z}$ be an isotopy of $MC$-cycles. $\mathfrak{P}(\tilde{Z})$ is a continuos piecewise smooth differential form  $\Omega^*([0,1]) \otimes \R[x,y]$.  The quantum master equation holds: 
\begin{equation} \label{master}
 d_t \mathfrak{P}(\tilde{Z}) + g_s \Delta \mathfrak{P}(\tilde{Z}) =0 .
\end{equation}
\end{proposition}
\begin{proof}
From (\ref{jump}), the continuity of $\mathfrak{P}(\tilde{Z})$ is equivalent to the relation
$$ \langle \tilde{P}, \tilde{w}_h(t_0^-) \times  \tilde{w}_{h'}(t_0^-)  \rangle =  \langle \tilde{P}, \tilde{w}_h(t_0^+) \times  \tilde{w}_{h'}(t_0^+)  \rangle   \pm 1 .$$
This follows from propriety (\ref{propagator-singularity}) of the propagator.
 
The quantum master  equation is an immediate consequence of the formula for $d \tilde P$
 of Lemma \ref{parameter-propagator-lemma}.
\end{proof}





\subsection{Open Gromov-Witten Potential}







To define the Open Gromov-Witten potential we need to require that the collection of $MC$-cycles $(Z_{\beta})_{\beta}$ satisfy  the factorization property that we now state. 

\subsubsection{Extended $MC$-cycles.}

To state the factorization property we need to extend the definition of nice $MC$-cycle. 

Given a closed one chain $w$, we say that $w'$ is $\epsilon$-close to  $w$ if we can write $w'= w'_1 + w'_0 $, with $w'_1$ $\epsilon$-close  in the $C^0$-topology to $w$ and $w'_0$
 $\epsilon$-close in the $C^0$-topology to a chain whose support is an union of points.

We say that a $MC$-cycle $Z$ is $\epsilon$-close to a finite set  $\{   w^j \}_j$ of closed one chains 
 if  each $w_h$ appearing in expansion of $Z$ is   
 $\epsilon$-close to an element of $\{   w^j \}_j$.

Given a sequence of nice $MC$-cycles  $Z^n \in \mathcal{Z}$  we write 
  $$\lim supp(Z^n) =  \{ w^j\}_j $$
  if for each $\epsilon>0$, $Z^n$ is $\epsilon$-close to $\{ w^j\}_j$ for $n \gg 0$, and  $\{ w^j\}_j$ is the minimal set with this property.

Write 
$ Z^n= \sum g_s^i Z^n_i .$
We say  that $Z$ has finite limit support if $Z_i$ has finite limit support for each $i$.

 An extended  nice $MC$-cycle consists in a sequence of nice $MC$-cycles 
   $Z= (Z^n)_n$  such that 
\begin{itemize}
    \item the limit support is finite;
    \item there exists $\tilde{Z}^{int, n}$ isotopy between $Z^n $ and $Z^{n-1}$ with 
 $$ \lim Supp( \tilde{Z}_{}^{int}) =   \lim Supp(Z_{})  .$$
\end{itemize}

Two extended $MC$-cycles $Z^{0,n}, Z^{1,n}$ are said equivalent if there exists a sequence of isotopies
$(\tilde{Z}^{ n})_n$ between them with 
$$ \lim Supp( \tilde{Z}) =   \lim Supp(Z^0) =  \lim Supp(Z^1)  .$$



\subsubsection{Factorization property}

The factorization  map 
$$\mathfrak{fact}_{\beta_1, \beta_2} : \mathcal{Z}_{\beta} \rightarrow \mathcal{Z}_{\beta_1} \otimes  \mathcal{Z}_{\beta_1} $$
is defined on the generators (\ref{generator2}) as 
\begin{multline}
\mathfrak{fact}_{\beta_1, \beta_2} ( (H,V,  ( w_{h,v}  )_{h \in H , v \in V }, ( \beta_v )_{v \in V} ) ) =  \\
\sum_{ \substack{V_1 , V_2| V_1 \sqcup V_2 = V, \beta(V_1)= \beta_1 ,  \beta(V_2)= \beta_2  \\   H_1,H_2| H_1 \sqcup H_2 = H } }   (V_1,H_1,  ( w_{h,v}  )_{h \in H_1 , v \in V_1 }, ( \beta_v )_{v \in V_1} ) \otimes  (V_2,H_2,  ( w_{h,v}  )_{h \in H_2 , v \in V_2 }, ( \beta_v )_{v \in V_2} ).
\end{multline}
In the formula we have used the notation   $\beta(V') = \sum_{v \in V'} \beta_v$ if $V' \subset V$.

We say that $(Z_{\beta})_{\beta}$ satisfy the factorization property if
\begin{equation} \label{nice-factorization-eq}
\mathfrak{fact}_{\beta_1, \beta_2}(Z_{\beta_1 + \beta_2}) =  Z_{\beta_1} \otimes Z_{\beta_2} .
\end{equation}
for each $\beta_1, \beta_2$.

\begin{proposition} \label{prop-factorization}
The $MC$-cycle $Z$ of Theorem \ref{main-theorem} can be taken such that it satisfies the factorization property (\ref{nice-factorization-eq}). 
\end{proposition}
This is proved in Section \ref{factorization-property-section}.



\subsubsection{Open Gromov-Witten Potential}

Set
\begin{multline} \label{potential-generator}
   \mathfrak{W}  (V, H,  ( w_{v,h} )_{v \in V,h \in H} ,  ( w^{ann}_h )_{h \in H}, ( \beta_v )_{v \in V}  ) = \\
   g_s \sum_{G \text{   connected}} \mathfrak{P}( G, (V, H,  ( w_{v,h} )_{v \in V,h \in H} ,  ( w^{ann}_h )_{h \in H}, ( \beta_v )_{v \in V}  )) 
\end{multline}

For a $MC$-cycle $Z_{\beta}$, $ \mathfrak{W} (Z_{\beta})$ is defined extending (\ref{potential-generator}) by linearity.   We have
$$\mathfrak{W}(Z)  \in \R[[g_s,x]].$$




Let $(Z_{\beta})$ be a collection of $MC$-cycles which satisfyies the factorization property.
Consider the Novikov Ring with formal variable $T$. Set 
$$  \mathfrak{P}((Z_{\beta})_{\beta}) = \sum_{\beta}  \mathfrak{P}(Z_{\beta}) T^{\omega(\beta)},$$
$$  \mathfrak{W}((Z_{\beta})_{\beta}) = \sum_{\beta}  \mathfrak{W}(Z_{\beta}) T^{\omega(\beta)}.$$

The factorization property implies 
$$ \mathfrak{P}(Z) =  \text{exp} (\frac{1}{g_s}\mathfrak{W}(Z)  ).$$

If $\tilde Z$ is an isotopy of $MC$-cycles which satisfies the factorization property,
 the master equation (\ref{master}) implies 
$$   d  \mathfrak{W} (\tilde{Z}) + \frac{1}{2} \{ \mathfrak{W} (\tilde{Z}),  \mathfrak{W} (\tilde{Z})   \}  + g_s \Delta   \mathfrak{W} (\tilde{Z}) =0 .  $$

\subsection{Numerical Invariants}

Let $W = \sum_{i \geq 0} g_s^{i} W_{i}$ a solution of the quantum master equation. 

Let $x_0 \in H^1(M, \Lambda_0)$ be a critical point of $W_0$:
$$ \frac{\partial W_0 }{\partial x} (x_0) =0 .$$
We assume that $x_0$ is isolated and not degenerated, i.e., the matrix 
$$  \partial_i \partial_j W_0(x_0)  $$
admits an inverse 
$$ \Delta_{i,j} := ( \partial_i \partial_j W_0(x_0))^{-1}  .$$



We considered decorated graphs $G$, where its vertex $v$ is decorated with an integer number 
$\chi_v \leq 1$, called its Euler Characteristic.   The Euler Characteristic of the graph is defined as $\chi(G) := \sum_v \chi_v - E^{in}(G)$.

To a decorated graph $G$ we associate  the contribution $Cont_G$ defined as follows: 
 we put the tensor $\partial^{H_v}W_{1-\chi_v}$ on the vertex $v$ and the propagator  $\Delta$ on every edge ($H_v$ denotes the set of half-edges attached to $v$).  $Cont_G$ is defined as the contraction of all these tensors.

Define the effective action 
$$ W^{eff} = \sum_{G}  \frac{g_s^{1-\chi(G)}}{\text{Aut}(G)}Cont_G .$$
The sum run over all the connected decorated graphs.








\subsubsection{Isotopies}

Let $\tilde{W}=  \sum_{i \geq 0} g_s^{i} \tilde{W}_{i}$ be an isotopy of solutions of quantum master equations.

The classical term involving  satisfies the equation:
\begin{equation} \label{QME-explicit-0}
\frac{\partial W_0 }{\partial t} +  \frac{\partial W_0 }{\partial x} \frac{\partial K_0 }{\partial y} =0.
\end{equation}
 $x_0$  dependence on $t$ since $W_0$ does it. From (\ref{QME-explicit-0}) we have: 
\begin{equation}    \label{critical-derivative}
    \frac{d x_0}{dt} = \frac{\partial K_0}{\partial y}
\end{equation}  
From  equations ( \ref{critical-derivative}) and  (\ref{QME-explicit-0}) we obtain the derivative of the propagator $\Delta$:
\begin{equation} \label{delta-derivative}
 \frac{d}{dt} \Delta = -  \frac{\partial K_0}{\partial y \partial x}(x_0)  \Delta.
\end{equation}

Equations  ( \ref{critical-derivative}) and  (\ref{QME-explicit-0}) yield also to
$$ \frac{d}{dt} [W_0(x_0)] = 0 ,$$
i.e., $W_0(x_0)$ is an invariant of the isotopy class of the solution of the classical Master Equation.

Using some standard graphs technique, we can extended the last claim to higher loops : 
\begin{proposition}
$W^{eff}$ does not depends on $t$.
\end{proposition}
\begin{proof}

We need to compute the total derivative $\frac{d}{dt}\tilde{W}_{i}^{eff}(x_0)$. There are three contributions :
\begin{itemize}
    \item the partial derivative of the tensors $\partial^{H_v}W_{\chi_v}$ used in the definition of $Cont_G$;
    \item the derivative of of $\Delta$ given by (\ref{delta-derivative});
    \item the derivative of $x_0$ given by (\ref{critical-derivative}).
\end{itemize}

In order to write explicitly the first contribution we use the quantum master equation. 
 Writing
$$ \tilde{W}(x,y) = W(x) + K(x,y)dt ,$$
 we can rewrite  the QME explicitly as 
\begin{equation} \label{QME-esplicit}
\frac{\partial W_n }{\partial t} + \sum_{m+l=n} \frac{\partial W_m }{\partial x} \frac{\partial K_l }{\partial y} + \frac{\partial^2 K_{n-1} }{\partial x \partial y} =0.
\end{equation}

Using (\ref{QME-esplicit}) we can write the contribution in terms of decorated graphs of the type we have used to define $Cont_G$ above, where the graph $G$ has a special vertex and a special edge. The special edge has one special half edge which  is attached to the special vertex.  On the special vertex we put the tensor  $ \partial^{H_v}K_{1-\chi_v}$, where $\partial^h= \partial_x$ if $h \in H_v$ is not special, and $\partial^h= \partial_y$ if $h \in H_v$ is special. To the special edge we put the identity tensor. 

The graphs appearing in the expansion are not necessarily stable. 
We have four types of unstable graphs:
\begin{enumerate}
\item the special vertex has $i=0$, and has no other half-edges attached besides the special one;
\item the special vertex has $i=0$, and it has exactly one half-edge attached besides the special one;
\item the special edge connects the special vertex to a vertex with $i=0$, with no other half-edges  attached;
\item the special edge connects the special vertex to a vertex with $i=0$, which has exactly one  more half-edge attached;
\end{enumerate}
The contribution of $(1)$ cancel with  (\ref{critical-derivative}). 
The contribution of $(2)$ cancel with  (\ref{delta-derivative}).
The contribution of $(3)$ vanishes using the definition of $x_0$.
The contribution of $(4)$ cancel with the contribution of the stable graphs: if we eliminate  the unstable vertex and connect the two edges attached to it we obtain a stable graph with opposite contribution.

\end{proof}

\section{Bulk Deformations}

We now include bulk deformations in our constructions.  They can be seen as a deformation of the four-chain $K$.  

 \subsection{Review of the construction of the $MC$-cycle}



Denote by $  \overline{\mathcal{M}}_{(g,V, D, (H)_v)}(\beta) $ 
the moduli  of curves of genus $g$ , homology class $\beta$, whose boundary components are labelled by $V$, internal marked points are  labelled by $D$ and boundary  marked points are labelled by $H_v$ for each $v$.

In \cite{OGW3}, to a decorated graph $(G,m) \in  \mathfrak{G}_l$, it is associated the moduli space of multi-curves
\begin{equation}    \label{multi-curve-interp0}
 \overline{\mathcal{M}}_{G,m} :=    \left( \prod_{c \in Comp(G)} \overline{\mathcal{M}}_{g_c,D_c, V_c, (H_v)_{v \in V_c} }(\beta_c)  \right)  / \text{Aut}(G,m) \times \Delta^l
 \end{equation}   
where $\Delta^l$ is the standard simplex  of dimension $l$.

The space $ \overline{\mathcal{M}}_{G,m} $ is endowed with a Kuranishi structure using an  argument analogous to the one used for the construction of the Kuranishi structures of stable pseudo-holomorphic curves. 

Recall that the standard procedure to construct a Kuranishi structure on moduli space of stable pseudo-homorphic curves uses an argument by induction on the strata. For each strata, we first use a  gluing argument to extend the Kuranishi structure to a neighbourhood of the substratas. After that the Kuranishi structure is extended inside the strata.  This method does not lead to a unique Kuranishi structure. However it has the fundamental property that two Kuranishi structures constructed in this way  are  isotopic, i.e., there exists a Kuranishi structure on the moduli spaces times $[0,1]$ whose restrictions to $0$ and $1$  agree with the starting Kuranishi structures.
 This isotopy is constructed using an analogous inductive argument.

Using an inductive argument similar to the one used in  the construction of the Kuranishi structures for stable pseudo-holomorphic curves, we can obtain a collection of Kuranishi structures $(\overline{\mathcal{M}}_{G, m })_{(G,m) \in \mathfrak{G}(\beta, \leq \kappa   )}$  such that:
\begin{itemize}
\item  forgetful compatibily holds;
\item the evaluation map 
$$\text{ev} : \overline{\mathcal{M}}_{G,m} \rightarrow L^{H(G) \setminus H_l} \times X^{D(G)}$$ 
is weakly submersive;
\item   the corner faces of $\overline{\mathcal{M}}_{(G,m)}$ are in bijection with  the graphs $\mathfrak{G}(G,m)$
$$  \mathfrak{G}(G,m ) \leftrightarrow  \{    \text{   corner faces of  } \overline{\mathcal{M}}_{(G,m)}    \} .$$
The corner face $ \overline{\mathcal{M}}_{G,m }( G',m',E') $ corresponding to  $(G',m',E') \in \mathfrak{G}(G,m)$ comes with an identification of Kuranishi spaces: 
\begin{equation} \label{corner-faces-multi-interp0}
 \overline{\mathcal{M}}_{G, m }(G' ,m', E') \cong   \delta_{E'}  \overline{\mathcal{M}}_{G',m'}.
\end{equation}
\end{itemize}
The Kuranishi spaces $\delta_{E'}  \overline{\mathcal{M}}_{G',m'}$ are defined as fiber products. See \cite{OGW3}.



Consider collections of perturbations
$( \mathfrak{s}_{(G,m)}^+ )_{(G,m) \in \mathfrak{G}(\beta, \leq \kappa   ) }$ 
of the collection of Kuranishi spaces
 $(\overline{\mathcal{M}}_{G,m} )_{G \in \mathfrak{G}(\beta, \leq \kappa   ) }$ such that 
\begin{itemize}
\item they are transversal to the zero section and small in order for constructions of virtual class to work;
\item they are compatible with the identification of Kuranishi spaces (\ref{corner-faces-multi-interp0});
\item 
  forgetful compatibility holds;
 \item  they are transversal  when restricted to $\delta_{E}  \overline{\mathcal{M}}_{G,m}$ for each $E \subset (E^{in}(G) \setminus E_l) \sqcup D(G)$ 
 \end{itemize}

Using these perturbations we define the collections of chains $( Z_{(G,m)}^+ )_{(G,m) \in \mathfrak{G}_*(\beta, \leq \kappa   ) }$
\begin{equation} \label{multi-curve-cycle-not-ab}
 Z_{(G,m)}^{not-ab, +} =  (\text{ev}_{G,{m}})_* (\mathfrak{s}_{G,{m}}^{-1}(0)) \in C_*(L^{H(G)} \times X^{D(G) })^{Aut(G,m)}  .
\end{equation}

The not abelian $MC$-cycle $Z^{not-ab}$ is obtained taking the fiber product
\begin{equation} \label{intersect-K-gen}
Z_{G,m}^{not-ab}=  Z_{G,m}^{not-ab,+} \times_{X^{D(G^+) }} (K^{D(G^+)}  ).
\end{equation}

From $Z^{not-ab}$ we can construct the abelian $MC$-cycle and the nice $MC$-cycle as in \cite{OGW3}.


\subsection{Bulk deformation of moduli spaces}

We now define the moduli space of multi-curves with internal punctures. The construction above can be straightforwardly adapted to this context. 

Define  the set of decorated graphs 
  with internal punctures  $\mathfrak{G}^{+}$: 
 an element    $G^+ \in \mathfrak{G}^+$ consists of an array 
$$(Comp, (V_c, P_c , D_c, \beta_c ,g_c )_c, (H_v)_v, E)$$ 
where 
\begin{itemize}
\item for each $c \in Comp(G^+)$,  $P_c$ is a finite set, called internal punctures.
\end{itemize}
All the other data are the same as in the definition of elements of $\mathfrak{G}$.
Set $P(G^+) = \sqcup_c P_c$.

Analogously to $\overline{\mathcal{M}}_{G,m}$, we can associate to $(G^+.m)$ the moduli space 
\begin{equation}    \label{multi-curve-interp0-bulk}
 \overline{\mathcal{M}}_{G^+,m} :=    \left( \prod_{c \in C(G)} \overline{\mathcal{M}}_{g_c,P_c, D_c, V_c, (H_v)_v }(\beta_c)  \right)  / \text{Aut}(G^+,m) \times \Delta_l
 \end{equation} 
where, for each component $c$, we have added to  the moduli space of pseudo-holomoprhic curves $\overline{\mathcal{M}}_{g_c,P_c ,D_c , V_c, H_c}(\beta_c)$ internal marked points labelled by $P_c$. 
We can equip  $(\overline{\mathcal{M}}_{G^+,m} )_{G^+ \in \mathfrak{G}^+(\beta, \leq \kappa   ) }$ with a Kuranishi structure with the following properties:  
\begin{enumerate}
\item  forgetful compatibility holds;
\item the evaluation map $\text{ev} : \overline{\mathcal{M}}_{G^+,m} \rightarrow L^{H(G^+) \setminus H_l} \times X^{D(G^+) \sqcup P(G^+)}$ is weakly submersive;
\item   the corner faces of $\overline{\mathcal{M}}_{(G^+,m)}$ are in bijection with  the graphs $\mathfrak{G}^+(G^+,m)$.
The corner face $ \overline{\mathcal{M}}_{G^+,m }( {G^+}',m',E') $ corresponding to  $({G^+}',m',E') \in \mathfrak{G}^+(G^+,m)$ comes with an identification of Kuranishi spaces: 
\begin{equation} \label{corner-faces-multi-bulk}
 \overline{\mathcal{M}}_{G^+, m }({G^+} ' ,m', E') \cong   \delta_{E'}  \overline{\mathcal{M}}_{{G^+}',m'}.
\end{equation}
\end{enumerate}



Consider collections of perturbations
$( \mathfrak{s}_{(G^+,m)}^+ )_{G^+ \in \mathfrak{G}^+(\beta, \leq \kappa   ) }$ 
of the collection of Kuranishi spaces
 $(\overline{\mathcal{M}}_{G^+,m} )_{G^+ \in \mathfrak{G}^+(\beta, \leq \kappa   ) }$ such that 
\begin{itemize}
\item they are transversal to the zero section and small in order for constructions of virtual class to work;
\item  they are compatible with the identification of Kuranishi spaces (\ref{corner-faces-multi-bulk});
\item 
  forgetful compatibility holds;
 \item  they are transversal  when restricted to $\delta_{E}  \overline{\mathcal{M}}_{G^+,m}$ for each $E \subset (E^{in}(G^+) \setminus E_l) \sqcup (D(G^+) \sqcup P(G^+) )$ 
 \end{itemize}

Using these perturbations we define the collections of chains $( Z_{(G^+,m)}^+ )_{(G^+,m) \in \mathfrak{G}_*^+(\beta, \leq \kappa   ) }$
\begin{equation} \label{multi-curve-cycle-not-ab-bulk}
 Z_{(G^+,m)}^{not-ab,+} =  (\text{ev}_{G^+,{m}})_* (\mathfrak{s}_{G^+,{m}}^{-1}(0)) \in C_*(L^{H(G^+)} \times X^{D(G^+) \sqcup P(G^+)})^{Aut(G^+,m)}  .
\end{equation}



The not abelian $MC$-cycle $Z^{not-ab}$ is obtained taking the fiber product
\begin{equation} \label{intersect-K-gen-bulk}
Z_{G^+,m}^{not-ab}=  Z_{G^+,m}^{not-ab,+} \times_{X^{D(G^+) \sqcup P(G^+)}} (K^{D(G^+)} \times A^{P(G^+)} ).
\end{equation}

From $Z^{not-ab}$ we can construct the  abelian $MC$-cycle and the nice $MC$-cycle as in \cite{OGW3}.


\subsection{Nice $MC$-cycles}

We now describe the nice $MC$-cycles with bulk deformations.

A generator with bulk deformations is defined by an array
\begin{equation} \label{generator2-bulk}
(V, H,  ( w_{v,h} )_{v \in V,h \in H} ,  ( w^{ann0}_h )_{h \in H},  ( w_h^{bulk0}  )_{h \in H}), (  \beta_v )_{v \in V}  )
\end{equation}
where $V$ and $H$ are finite sets, $\beta_v \in \Gamma$,   $ \{ w_{v,h} \}_{h \in H}$ are closed one dimensional chain on $M$ on the homology class $\partial \beta_v$ , which are close in the $C^0$-topology.  
 $ \{ w^{ann}_h \}_{h \in H}$  are closed one dimensional chains on $M$  which are close to $w^{ann}$ in the $C^0$-topology.  
$( w^{bulk0}_h )_{h \in H}$   are closed one dimensional chains on $M$  which are close to $A \cap L$ in the $C^0$-topology.  

We assume the support property  (\ref{support-property-nice-cycles}).


Denote by $\mathfrak{Gen}(\beta,A)^{\dagger}$ the set of objects (\ref{generator2-bulk}) such that $\sum \beta_v = \beta$, modulo the obvious equivalence relation.


Set $w_h= \sum w_{v,h} + w^{ann0}_h+ w_h^{bulk0}$.
Define $\mathfrak{Gen}(\beta,A) $ as the subset of $ \mathfrak{Gen}(\beta,A)^{\dagger}$ obtained imposing the extra condition  
\begin{equation} \label{generator-transversality2-bulk}
w_{h} \cap w_{h'} = \emptyset \text{    if    } h \neq h'.
\end{equation}

Let $\mathcal{Z}^{A}_{\beta} $ be the set of formal power series
$$ \sum_i g_s^{k_i} \mathfrak{b}^{l_i}  Gen_i$$
with $Gen_i \in \mathfrak{Gen}(\beta,A) $, $k_i , l_i \rightarrow \infty$, and $k_i + l_i+  |V_i| \geq 0$. 

\subsubsection{Isotopies}

To define  isotopies of nice-$MC$ cycles  with bulk deformations  we consider objects 
\begin{equation} \label{generator2-iso-bulk}
(V, H,  ( \tilde{w}_{v,h} )_{v \in V,h \in H} ,  ( \tilde{w}_{h}^{ann0} )_{h \in H}   ,  ( \tilde{w}_{h}^{bulk0} )_{h \in H}, (  \beta_v )_{v \in V} , [a,b]   )
\end{equation}
where $[a,b] \subset [0,1]$,  $ \tilde{w}_{h,v} : F_v \times [a,b] \rightarrow M$ for some one dimensional compact manifold $F_v$. Set
$ \tilde{w}_{h}^{\flat} = \sum_v \tilde{w}_{h,v} $.

Let $\tilde{\mathfrak{Gen}}(\beta, A)$ be set of objects (\ref{generator2-iso-bulk}) modulo the obvious equivalence relation.

An isotopy of of nice-$MC$ cycles $\tilde{Z} $ is defined by formal power series 
$$\tilde{Z}= \sum_i r_i g_s^{k_i} \mathfrak{b}^{l_i} (V_i, H_i,  ( \tilde{w}_{v,h,i} )_{v \in V_i,h \in H_i} ,  ( \tilde{w}_{h,i}^{ann0} )_{h \in H_i}   , ( \tilde{w}_{h,i}^{bulk0} )_{h \in H_i},   (  \beta_v )_{v \in V_i} , [a_i,b_i]   ) ,$$
with $r_i \in \Q$, $k_i, l_i \rightarrow \infty$, and $k_i + l_i+ |V_i| \geq 0$. 

The formal power series are considered modulo gluing of objects (\ref{generator2-iso-bulk}).

$\tilde{Z}$ defines a one parameter family of nice $MC$-cycles $(Z^t)_t$ with  $Z^t \in \mathcal{Z}^{A}_{\beta}$ discontinuous for a finite number of times. The discontinuity at the time $t_0$ is obtained by the formula:
\begin{equation} \label{jump2-bulk}
 Z^{t_0^+} - Z^{t_0^-} =  \pm r' g_s^{k'}   \mathfrak{b}^{l'}  (V', H',  ( w_{v,h}' )_{v \in V',h \in H'} ,({w_{h}^{ann0}}')_{h \in H'}, ({w_{h}^{bulk0}}')_{h \in H'},( \beta_v' )_{v \in V'}  )
\end{equation}
As in the case without bulk deformations we have $r' = r_i$, $H'= H \setminus \{ h_1,h_2\} $ and the sign is determined by the sign of the crossing.
In the case 
$\tilde{w}_{ h_1,i}$ crosses $\tilde{w}_{ h_2,i}$   or 
$\tilde{w}_{h_1,i}$ crosses $\tilde{w}^{ann}_{ h_2,i}$ we have $l'=l_i,k' = k_i+1$ and the data on the RHS are the same without bulk deformations.
Here we have two extra cases to consider:
\begin{itemize}
     \item  If $\tilde{w}_{h_1,i}$ or $\tilde{w}^{ann}_{ h_1,i}$ crosses $\tilde{w}^{bulk}_{ h_2,i}$, 
     we have $V'=V_i , k'= k_i, l' = l_i+1 $;
    \item If  $\tilde{w}^{bulk}_{ h_1,i}$ crosses $\tilde{w}^{bulk}_{ h_2,i}$, 
    we have $V'=V_i, k'= k_i-1, l' = l_i+2 $.
\end{itemize}
Denote by  $\tilde{\mathcal{Z}}_{\beta}^{A}$ the set of isotopies of nice $MC$-cycles in the class $\beta$.


\subsection{Partition Function}

Let 
$$G =(V(G), ( H_v )_{v \in V(G)}, E(G), (\beta_v )_{v \in V(G)} , V^{ann0}, V^{bulk0})$$ 
be a decorated graph, where  $V^{ann0},  V^{bulk0} \subset V(G)$ are special sets of vertices. 
 For each $v \in V^{ann0} \sqcup V^{bulk0}$ we assume $\beta_v=0$ and $|H_v|=1$. We assume that there are no other area zero vertices besides $ V^{ann0 }  \sqcup V^{bulk0} $ and that there exists an identification $V(G) \setminus (V^{ann0}  \sqcup V^{bulk0}   )= V$ compatible with the data $(\beta_v)_v$.
Set 
$w_{h,v_h} = w_h^{ann0} $ if $v_h \in V^{ann0}$,
$w_{h,v_h} = w_h^{bulk0} $ if $v_h \in V^{bulk0}$.

The contribution to the partition function associated to $G$ is described  by the formula
\begin{multline} \label{contribution-graphs-bulk}
    \mathfrak{P}( G, (V, H,  ( w_{v,h} )_{v \in V,h \in H} ,  ( w^{ann0}_h )_{h \in H}, ( w^{bulk0}_h )_{h \in H}, (  \beta_v )_{v \in V}  )) =  \\
    \sum_{s}   \frac{g_s^{|E^{in}(G)| - | V^{bulk0}| } \mathfrak{b}^{| V^{bulk0}|}  }{|\text{Aut}(G)|} \prod_{\{ h,h' \} \in E^{in}} \langle P,w_{s(h),v_h}  \times w_{s(h'),v_{h'}}  \rangle   \times \prod_{ \{ h \} \in E^{ex}} \langle \sum_i x_i \alpha_i, w_{s(h),v_h} \rangle.
\end{multline}
where the sum is made  on the set of bijective maps $s : \sqcup_v H_v  \rightarrow H$.

The partition function $\mathfrak{P}( Z)$ associated to a nice $MC$-cycle  $Z \in \mathcal{Z}_{\beta}^A$ is obtained as formal power series adding all the contributions (\ref{contribution-graphs-bulk}).  $\mathfrak{P}( Z)$ admits an expansion of formal power series 
$$ \mathfrak{P}( Z) = \sum r_i g_s^{ k_i } \mathfrak{b}^{l_i} p_i(x)$$
where $k_i , l_i \rightarrow \infty$, $k_i + l_i+ N_{\beta} \geq 0$ and 
$p_i(x) \in \R[[(x_j)_j]] $.

As in the case without balk deformation,
the $MC$-cycle $ (Z_{G^+,m})_{G^+,m}$ which we obtarin from the moduli space of mutlti-curves depends on the varies choices we made. Two different choices lead  to isotopic $MC$-cycles. To an isotopy $ (\tilde{Z}_{G^+,m})_{G^+,m}$ of  $MC$-cycles it is associated an isotopopy of quantum master equation $\mathfrak{P} (\tilde{Z})$ :
$$ \frac{d}{dt} \mathfrak{P} (\tilde{Z}) + g_s \Delta \mathfrak{P} (\tilde{Z})=0 .$$


Denote by $Z^{K,A}_{\beta}$ the nice $MC$-cycle obtained from the moduli space of multi-curves with four chain $K$ and bulk deformation $A$ and 
set 
$$ \mathfrak{P} (\beta,  K, A)= \mathfrak{P} (Z^{K,A}_{\beta}) .$$
$ \mathfrak{P} (\beta, K, A)$ is well defined up to isotopy.

From the definition it is immediate to check that
$$  \mathfrak{P} (\beta, K+ r A, A)  (g_s, \mathfrak{b}) =  \mathfrak{P} (\beta, K, A)  (g_s, \mathfrak{b} + r g_s ).  $$

\subsubsection{Bulk deformations with more chains}
Fix a finite set of four-chains $A_1,A_2,...,A_l$ which intersect transversely.
Using the same procedure above
we can  consider bulk deformations by $A_1,A_2,...,A_l$. We introduce formal variables  $\mathfrak{b}_1,\mathfrak{b}_2,..., \mathfrak{b}_l$. 
We need to consider decorated graphs with internal punctures $P_1,P_2,...,P_l$. It is straightforward to extend the construction above to define  moduli space of multi-curves and obtain the associated $MC$-cycle  $Z^{K,A_1,...,A_l}_{\beta}$. 

The space of nice $MC$-cycle is generated by objects
\begin{equation} \label{generator2-bulk-more}
(V, H,  ( w_{v,h} )_{v \in V,h \in H} ,  ( w^{ann0}_h )_{h \in H},  ( w_h^{bulk0,j}  )_{h \in H, 1 \leq j \leq l}, (  \beta_v )_{v \in V}  ).
\end{equation}
The definition of the partition function $ \mathfrak{P} (Z)$ is extended straightforwardly.

Set 
$$ \mathfrak{P} (\beta, K, A_1,...,A_l)= \mathfrak{P} (Z^{K,A_1,...,A_l}_{\beta}) .$$

From the definition we have that
$$  \mathfrak{P} (\beta, K+ \sum_i r_i A_i, A_1,...,A_l)  (g_s, \mathfrak{b}_1 , ...,  \mathfrak{b}_l ) =  \mathfrak{P} (\beta, K, A)  (g_s, \mathfrak{b}_1 + r_1 g_s,..., \mathfrak{b}_l + r_l g_s ).  $$

\section{Factorization property} \label{factorization-property-section}


In the construction of \cite{OGW3}, we fixes a class $\beta \in H_2(X,L)$ and considered Multi-Curves in the class $\beta$ without imposing any constrain with multi-curves in the classes $\beta'$ with $\omega(\beta') < \omega(\beta) $. If $\beta= \beta_1 + \beta_2$, the product of the moduli space of multi-curves of classes $\beta_1$ and $\beta_2$, is part of the moduli of multi-curves in class $\beta$. This fact should   leads to a relation between the element $Z_{\beta}$ and the product $Z_{\beta_1} \times Z_{\beta_2}$. However in construction of \cite{OGW3} does not give directly a strict equality, but only up to isotopies. We now explain precisely what this means, and how to get the strict equality that we need.

\subsection{Factorization map}

In order to prove the factorization property it is necessary to realize the factorization property on the moduli space of multi-curves. In order to achieve this, our first step is to extend the factorization property  (\ref{nice-factorization-eq}) to the (not nice) $MC$-cycles.


For $\beta_1, \beta_2 \in \Gamma$, define $\mathfrak{G}_{l}(\beta_1,\beta_2) = \mathfrak{G}_{ l}(\beta_1) \times \mathfrak{G}_{ l}(\beta_2)$, and $\mathfrak{G}(\beta_1,\beta_2)= \sqcup_l \mathfrak{G}_l(\beta_1,\beta_2)$. The operator $\delta_e$ extends straightforwardly to $\mathfrak{G}(\beta_1,\beta_2)$.

Using this we define a version  of $MCH$, where the space of $MC$-chains  $\mathcal{C}_{\beta_1,\beta_2}$ are given by collections of chains $\{ C_G \}_{G \in \mathfrak{G}(\beta_1,\beta_2)}$ .  We shall denote the $MC$-cycles as $\mathcal{Z}_{\beta_1,\beta_2}$.


We need to consider two main operations:
\begin{itemize}
\item The factorization map :
\begin{equation} \label{restriction}
 \mathbf{fact}_{\beta_1,\beta_2} :  \mathcal{C}_{\beta} \rightarrow  \mathcal{C}_{(\beta_1,\beta_2)}.
\end{equation}
given by 
\begin{equation} \label{fact-map}
 \mathbf{fact}_{\beta_1,\beta_2}(C)_{(G_1,\{ E_{i,1}  \}_{0 \leq i \leq l}) , (G_2, \{ E_{i,2}  \}_{0 \leq i \leq l})} := C_{G_1 \sqcup G_2,\{ E_{i,1} \sqcup E_{i,2} \}_{0 \leq i \leq l}}.
\end{equation}


\item The  product of $MC$-chains:
\begin{equation} \label{cup-formula}
(C^1\boxtimes C^2)_{(G^1, m^1) , (G^2, m^2) } := \sum_{0 \leq r \leq l}  C^1_{G^1, m^1_{[0,r]}} \times  C^2_{G^2, m^2_{[r,l]}}. 
\end{equation}
A direct computation shows that $\cup$ is compatible with $\hat{\partial} $:
$$  \hat{\partial}  (C^1 \boxtimes C^2) =  \hat{\partial} (C^1) \boxtimes C^2  + C^1 \boxtimes  \hat{\partial} (C^2)  .$$

\end{itemize}



\subsection{Moduli spaces} \label{factorization-moduli-section}

We now define moduli space of multi-curves associated to decorated graphs $ \mathfrak{G}(\beta_1 ,\beta_2)$ instead of  $ \mathfrak{G}(\beta)$.

To an element $(G_1,m_1),(G_2,m_2) \in  \mathfrak{G}(\beta_1 ,\beta_2) $ we  associate the moduli space 
\begin{equation}    \label{multi-curve-interp0-fact}
 \overline{\mathcal{M}}_{(G_1,m_1),(G_2,m_2)} :=    \left( \prod_{c \in C(G)} \overline{\mathcal{M}}_{g_c,D_c, V_c ,(H_v)_{v \in V_c} }(\beta_c)  \right)  / \text{Aut}((G_1,m_1),(G_2,m_2)) \times \Delta^l  .
 \end{equation} 
Observe that $ \overline{\mathcal{M}}_{(G_1,m_1),(G_2,m_2)} $ differs from 
$\overline{\mathcal{M}}_{ (G_1 \sqcup G_2 , m_1 \sqcup m_2  )}$
by the  automorphism group:  
\begin{equation} \label{space-fact}
\overline{\mathcal{M}}_{ (G_1,m_1) , (G_2, m_2)}/  (\text{Aut}(G_1 \sqcup G_2 , m_1 \sqcup m_2  ) / \text{Aut}((G_1,m_1) , (G_2, m_2)) ) \cong  \overline{\mathcal{M}}_{ (G_1 \sqcup G_2 , m_1 \sqcup m_2  )} .
\end{equation}
Using an inductive argument on the strata analogous to the one we have used for as the moduli spaces $\overline{\mathcal{M}}_{G, m }$, the moduli $ \overline{\mathcal{M}}_{(G_1,m_1),(G_2,m_2)} $  can be equipped  with a Kuranishi structure such that
\begin{enumerate}
\item $ \overline{\mathcal{M}}_{(G_1,m_1),(G_2,m_2)} $ is forgetful compatible;
\item the evaluation map 
$$\text{ev} :  \overline{\mathcal{M}}_{(G_1,m_1),(G_2,m_2)} \rightarrow L^{H(G_1 \sqcup G_2) \setminus (H_{1,l} \sqcup H_{2,l})} \times X^{D(G_1)\sqcup D(G_2)}$$ 
is weakly submersive;
\item   the corner faces of  $ \overline{\mathcal{M}}_{(G_1,m_1),(G_2,m_2)} $ are in bijection with  the graphs $\mathfrak{G}((G_1,m_1),(G_2,m_2))$.
\end{enumerate}
Here $\mathfrak{G}(G_1,m_1,G_2,m_2)$ is defined as the set of $ (G_1',E_1', m_1'),(G_2',E_2',m_2')$ such that $G_1'/E_1'=G_1$, $G_2'/E_2'=G_2$ with $m_1,m_2$ identified with $m_1',m_2'$ throughout the isomorphisms.

Two Kuranishi structure constructed in this way will be isotopic.

Given a four chain $K$, analogously to the Kuranishi spaces  $\overline{\mathcal{M}}_{G, m }^K$ introduced in \cite{OGW3},  we  set
$$ \overline{\mathcal{M}}_{(G_1,m_1),(G_2,m_2)}^K =  \overline{\mathcal{M}}_{(G_1,m_1),(G_2,m_2)} \times_{X^{D(G_1) \sqcup D(G_2)}} K^{D(G_1) \sqcup D(G_2)} .$$


Using a Kuranishi Structure on $ \overline{\mathcal{M}}_{(G_1,m_1),(G_2,m_2)} $ as above we equip $\overline{\mathcal{M}}_{(G_1,m_1),(G_2,m_2)}^K$   with a Kuranishi structure such that
\begin{enumerate}
\item the  forgetful compatibility holds;
\item the evaluation map 
$$\text{ev} :  \overline{\mathcal{M}}_{(G_1,m_1),(G_2,m_2)}^K \rightarrow L^{H(G_1 \sqcup G_2) \setminus (H_{1,l} \sqcup H_{2,l}} \times X^{D(G_1)\sqcup D(G_2)}$$ 
is weakly submersive;
\item   To each  $(G_1',E_1', m_1'),(G_2',E_2',m_2')  \in \mathcal{G}(G_1,m_1,G_2,m_2)$ corresponds a corner face $ \overline{\mathcal{M}}_{(G_1,m_1),(G_2,m_2) }^K((G_1',E_1', m_1'),(G_2',E_2',m_2')) $ of $ \overline{\mathcal{M}}_{(G_1,m_1),(G_2,m_2) }^K$ which comes with an identification of Kuranishi spaces
\begin{equation} \label{corner-faces-multi-interp0-fact}
 \overline{\mathcal{M}}_{(G_1,m_1),(G_2,m_2) }^K( (G_1',E_1', m_1'),(G_2',E_2',m_2'))  =   \delta_{E'_1}  \delta_{E'_2}  \overline{\mathcal{M}}_{(G_1',m_1'),(G_2',m_2')}^K.
\end{equation}
\end{enumerate}

As in \cite{OGW3}, we consider collections of  perturbations 
$( \mathfrak{s}_{(G_1,m_1), (G_2,m_2)} )_{ G_1,m_1), (G_2,m_2) \in \mathfrak{G}_*(\beta_1,\beta_2, \leq \kappa)}$ 
 of the collection of Kuranishi spaces $(\overline{\mathcal{M}}_{G,m}^K )_{ G_1,m_1), (G_2,m_2) \in \mathfrak{G}_*(\beta_1,\beta_2, \leq \kappa)}$ such that
\begin{enumerate}
\item are transversal to the zero section and small in order for constructions of virtual class to work;
\item are compatible with the identification of Kuranishi spaces (\ref{corner-faces-multi-interp0-fact});
\item 
 the forgetful compatibility   holds.
\end{enumerate}

Set
$$ Z_{(G_1,m_1), (G_2,m_2)} = \text{ev}_*( \mathfrak{s}_{(G_1,m_1), (G_2,m_2)}^{-1}(0)) .$$
The $MC$-cycle $Z$ depends on the many choices we made. The same  cobordism argument used in \cite{OGW3} implies that two $MC$-cycles constructed in this way are isotopic.

\subsection{Proof of the factorization property}

We have the following main Lemma:
\begin{lemma} \label{lemma-factorization-fondamental}
The $MC$-cycles $  \mathbf{fact}_{\beta_1,\beta_2} (Z_{\beta})$ and $ Z_{\beta_1} \boxtimes Z_{\beta_2} $ are isotopic. 
\end{lemma}

\begin{proof}

In order to prove the Lemma we need to realize the $MC$-cycles $  \mathbf{fact}_{\beta_1,\beta_2} (Z_{\beta})$ and $ Z_{\beta_1} \boxtimes Z_{\beta_2} $  as virtual fundamental class of the moduli space of multi-curves $ \overline{\mathcal{M}}_{(G_1,m_1),(G_2,m_2) }^K$ introduced above .

Fix a a Kuranishi structure on $ \overline{\mathcal{M}}_{G,m}$ constructed by inductive argument and pertubations $\{ \mathfrak{s}_{(G,m)} \}_{(G,m) \in    \mathfrak{G}(\beta, \leq \kappa), }$ as in \cite{OGW3}. Denote by $ Z_{\beta}$ its associated $MC$-cycle.

 Using relation (\ref{space-fact}),
the Kuranishi structure on $ \overline{\mathcal{M}}_{G,m}$ defines a Kuranishi structure on  $ \overline{\mathcal{M}}_{(G_1,m_1),(G_2,m_2) }^K$ for $G= G_1 \sqcup G_2, m= m_1 \sqcup m_2$, of the type considered above.

Moreover the collection of perturbations $( \mathfrak{s}_{(G,m)} )_{(G,m) \in   \mathfrak{G}(\beta, \leq \kappa) }$ yields to a collection of  perturbations 
$( \mathfrak{s}_{(G_1,m_1), (G_2,m_2)} )_{ G_1,m_1), (G_2,m_2) \in \mathfrak{G}_*(\beta_1,\beta_2, \leq \kappa)}$ 
 of the collection of Kuranishi spaces $(\overline{\mathcal{M}}_{G,m}^K )_{ (G_1,m_1), (G_2,m_2) \in \mathfrak{G}_*(\beta_1,\beta_2, \leq \kappa)}$
 with the properties stated above.  The associated  $MC$-cycle $Z_{\beta_1,\beta_2}$ realize the factorization map applied to $Z_{\beta}$:
\begin{equation} \label{factorization-realization}
     Z_{\beta_1,\beta_2} = \mathbf{fact}_{\beta_1,\beta_2} (Z_{\beta}) .
\end{equation}



For each integer $l$ fix a  decomposition of the standard $l$-simplex $\Delta^l= [0,1,...,l]$ 
\begin{equation} \label{fact-symplex}
     \Delta^l = (\bigsqcup_{0 \leq v \leq l} [0,1,...,v] \times [v,v+1 ...,l] )/ \sim 
\end{equation}
where the face $\partial_{v+1} [0,1,...,v, v+1] \times [v+1, ...,l]$ of  $[0,1,...,v, v+1] \times [v+1, ...,l]$ is identified with the face $ [0,1,..., v] \times \partial_{v} [v, ...,l]$ of  $[0,1,...,v] \times [v, ...,l]$. We assume that the decomposition induced on each boundary face of $\Delta^l$  is compatible with the decomposition of $\Delta^{l-1}$ .

Using (\ref{fact-symplex}) we obtain  an identification of moduli spaces 
\begin{equation} \label{moduli-factorization}
 \overline{\mathcal{M}}_{ (G_1,\{ E_{i,1}  \}) , (G_2, \{ E_{i,2}  \})} = ( \bigsqcup_{0 \leq v \leq l}  \overline{\mathcal{M}}_{(G_1,  \{ E_{i,1}  \}_{0 \leq i \leq v})} \times \overline{\mathcal{M}}_{(G_2,  \{ E_{i,2}  \}_{v \leq i \leq l})} ) /\sim.
\end{equation}
Define on $ \overline{\mathcal{M}}_{ (G_1,\{ E_{i,1}  \}) , (G_2, \{ E_{i,2}  \})} $ the Kuranishi structure and perturbation induced by the right side of (\ref{moduli-factorization}).  
The associated  $MC$-cycle $Z_{\beta_1,\beta_2}$ is the product of $Z_{\beta_1} $ and $ Z_{\beta_2}$:
\begin{equation} \label{product-realization}
     Z_{\beta_1,\beta_2} = Z_{\beta_1} \boxtimes Z_{\beta_2}  .
\end{equation}

The perturbations of the moduli spaces $(\overline{\mathcal{M}}_{G,m}^K )_{ G_1,m_1), (G_2,m_2) \in \mathfrak{G}_*(\beta_1,\beta_2, \leq \kappa)}$ we have used to construct the $MC$-cycles (\ref{factorization-realization}) and (\ref{product-realization}) are both of the type considered in subsection \ref{factorization-moduli-section}. Hence they are isotopic.






The Lemma follows from the fact that the Kuranishi structures and perturbations of 
$ \overline{\mathcal{M}}_{ (G_1,\{ E_{i,1}  \}) , (G_2, \{ E_{i,2}  \})} $ used  (\ref{space-fact}) and (\ref{moduli-factorization}) are isotopic. This fact follows from general argument, as in \cite{OGW3}.

\end{proof}



In order to construct a nice $MC$-cycle which satisfies the factorization property we are going to construct the nice $MC$-cycle directly from the moduli space of multi-curves, restricting the class of allowed perturbations. The factorization property will follow from  Lemma \ref{lemma-factorization-fondamental} adapted to these class of perturbations.







We assume that the Kuranishi structure have the following property:
\begin{itemize}
\item The obstruction bundle of  $ \overline{\mathcal{M}}_{G,m}$ contains the obstruction bundle of  $ \overline{\mathcal{M}}_{cut_{E}(G)}$, for each $E \subset E(G)$.
\end{itemize}
Observe that this condition is compatible at the boundary. 

For each graph $(G,m)$, from  $ \mathfrak{s}_{cut_{E(G)}G}$ we obtain  a perturbation  $ \mathfrak{s}_{G,m}^{\dagger} $ of $ \overline{\mathcal{M}}_{G,m}$.
We assume 
\begin{itemize}
\item if $G= \sqcup_i G_i$ is the decomposition of $G$ in connected components, we require 
$$ \mathfrak{s}_{G}  = \prod_i  \mathfrak{s}_{G_i} ;$$ 
\item $ \mathfrak{s}_{G,m} $ is close on the $C^0$-topology to  $ \mathfrak{s}_{G,m}^{\dagger} $.
\end{itemize}
In particular the second assumptions implies that the zero set of  $ \mathfrak{s}_{G,m} $ is close to the zero set of  $ \mathfrak{s}_{G,m}^{\dagger} $ outside a small neighborhood of the boundary.

We can adapt the inductive argument of \cite{OGW3} used to construct the perturbation of the moduli space imposing the above conditions . Note that these  conditions  are compatible with the boundary conditions used in the argument of  \cite{OGW3}. 

In this way we obtain the nice $MC$-cycle $Z_{\leq \kappa}$ for the $MC$-complex truncated to  $ \mathfrak{G}(\beta,  \leq \kappa   )$, for any $\kappa$ . Adapting Lemma \ref{lemma-factorization-fondamental} to these class of perturbations, we can prove that  $Z_{\leq \kappa}$ satisfies the factorization property. 

To construct  a full nice $MC$-cycle  we can apply a standard extension argument (see \cite{Boundary-States}). 

The same argument apples to construct isotopies of nice $MC$-cycles satisfying the factorization property. 


\section{Appendix: Abelian Chern-Simons Propagator}

In this section we review the construction of the Abelian  Chern-Simons Propagator. 


Let $M$ be a compact oriented three-manifold.
Let  $\mathcal{A} < \Omega^*(M)$ be a finite dimensional vector subspace of the space of closed differential forms of $M$ such that the inclusion in the cohomology of $M$ defines an isomorphism:
$$   \mathcal{A} \xrightarrow{\,\raisebox{-0.4ex}{$\sim$}\,} H^*(M) . $$
Let $\alpha_i ,\beta_i \in  \mathcal{A} $, such that $([\alpha_i] ,[\beta_i] )_i$ is a  symplectic basis of $H^*(M)$.




Denote by  $\text{Conf}_2(M)$ the configuration space of two points of $M$ , which is the differential geometric blow-up of $M \times M$ along the diagonal $Diag \subset M \times M$. 
$\text{Conf}_2(M)$  is a manifold with boundary characterized by the following two properties:
\begin{itemize}
    \item The interior of $\text{Conf}_2(M)$ is given by  $ M \times M \setminus Diag$;
    \item the boundary of $\text{Conf}_2(M)$ is  the sphere bundle of the normal bundle to the diagonal 
$$ \partial  \text{Conf}_2(M) = S(N(Diag \subset M \times M) .$$
\end{itemize}
The manifold $\text{Conf}_2(M)$ comes with natural maps $ pr_i :  \partial \text{Conf}_2(M)  \rightarrow M $, $i=1,2$ . 


We have also a switch diffeomorphism 
$$sw : \text{Conf}_2(M)  \rightarrow \text{Conf}_2(M) ,$$ 
which is inducted from the switch homomorphism on $M \times M$.

Denote by $H^2(\partial \text{Conf}_2(M) , \Z )^-$  the cohomology anti-invariant under the reverse diffeomorhism. Observe that 
$$H^2(\partial \text{Conf}_2(M) , \Z )^- = H^2(S^2, \Z)  =  \Z .$$

\begin{lemma}  \label{construction-propogator}
Let  $\mathcal{A} < \Omega^*(M)$ be as above. There exists a differential form $P \in \Omega_2(\text{Conf}_2(L))$ such that 
\begin{enumerate}
\item $dP = \sum_i (\alpha_i \times \beta_i  + \beta_i \times \alpha_i )$  \label{propagator-differential}
\item
$  [i_{\partial}^*( P)] \in H^*(\partial \text{Conf}_2(M) )^-   $ is the generator of $H^*(\partial \text{Conf}_2(M) , \Z )^-$.
\label{propagator-singularity}
Here $i_{\partial}: \partial \text{Conf}_2(M) \rightarrow \text{Conf}_2(M) $ denotes the inclusion map of the boundary;

\item $ \text{sw}^*(P) = -P  $, where $sw : \text{Conf}_2(M)  \rightarrow \text{Conf}_2(M) $ is the switch homomorphism, induced from the switch homomorphism on $M \times M$.
\item $\int_{\text{Conf}_2(M)} P \wedge \pi_1^*(\alpha') \wedge \pi_2^*(\alpha'') =0 $  for each   $ \alpha', \alpha'' \in \mathcal{A} .$
\end{enumerate}
\end{lemma}

Instead to consider $P$ as a differential forms on $\text{Conf}_2(L)$ we can think $P$ as an element of $\Omega^2(L \times L \setminus Diag)$, with singularity along the diagonal.
The second condition implies that
$\text{lim}_{\epsilon \rightarrow 0} \int_{S^2_{\epsilon}} P =  1 $          
where $S^2_{\epsilon}$ is a family of spheres of radius $\epsilon$ surrounding the diagonal of $M \times M$. 

Using the propagator $P$ we  define the operator 
$$ \mathbf{P} : \Omega^*(M) \rightarrow \Omega^*(M) $$ 
$$ \mathbf{P}(\omega) = \pi_* (P \wedge \pi_2^* (\omega) )  $$ 
which satisfies the relation
\begin{equation} \label{diff-P}
d \mathbf{P} - \mathbf{P} d = Id - \pi_{\mathcal{A}} .
\end{equation}

Since the construction of the propagator $P$ of this paper is slightly simpler than the one of \cite{CS} (see remark below), we provide a quick construction of it. 
Let $\eta \in \Omega^2(\partial Conf_2(L))$ be a closed two form such that 
$ [i_{\partial}^*( \eta)] = \omega \otimes 1 \in H^2(\partial \text{Conf}_2(M) )   $.
Extend $\eta$ to a closed differential form $\eta'$ on a collar $T$ of $\partial C_2(M)$. Let $\rho$ a cut off function $T$.  The differential form $d (\rho \eta') $ can be considered as a smooth differential form on $M \times M$. Its homology class is  $ [d (\rho \eta')] = PD (Diag) \in H^3(M \times M)$. Hence there exists  $\psi \in \Omega^2(M \times M )$ such that 
$ d (\rho \eta')  = \sum_i (\alpha_i \times \beta_i  + \beta_i \times \alpha_i ) + d \psi  . $
Set $P'= \rho \eta' - \psi $. $P'$ satisfies   conditions $(1), (2)$ and $(3)$ of the lemma.  Finally, to obtain also  condition (4) we add to $P'$ a suitable element of $\mathcal{A} \otimes \mathcal{A}$.

\begin{remark}
The construction of $P$ was provided in \cite{CS} in the case of usual perturbative Chern Simons theory.  However, in this paper we consider an analogous of the \emph{point splitting} regularization of the abelian Chern-Simons perturbation expansion which is naturally related to Open Gromov-Witten theory. For  this reason, in \cite{CS}  it is required  a more restrictive condition on $i_{\partial}^*( P)$ than condition   (\ref{propagator-singularity}) above and it requires a choice of a connection on the tangent space of $M$. 
\end{remark}

\subsubsection{One parameter version} \label{parameter-propagator-section}
The propagator $P$ is not uniquely determinated but depends on the choice of $\mathcal{A}$ and $\eta$.   After these choices are made $P$ is uniquely determined up to  $ d \Omega^1(M \times M)$.
 We shall show that different choices yield to isotopic propagators.

Consider a one parameter family $ ( \mathcal{A}_t )_{t \in [0,1]}$ of finite dimensional vector subspaces of closed differential forms of $M$ such that the inclusion in the cohomology  
$   \mathcal{A}_t \xrightarrow{\,\raisebox{-0.4ex}{$\sim$}\,} H^*(M) $
is an isomorphism for each $t$.

\begin{lemma}
 To $ \{ \mathcal{A}_t \}_{t \in [0,1]}$ we can associate a finite vector space $\tilde{ \mathcal{A}}$ of closed differential forms on $M \times [0,1]$, 
such that  
$$\tilde{ \mathcal{A}} \xrightarrow{\,\raisebox{-0.4ex}{$\sim$}\,} H^*(M \times [0,1])$$
 is an isomorphism and $i_t^*(\tilde{ \mathcal{A}}) = \mathcal{A}_t $ for each $t \in [0,1]$ (here $i_t: M \hookrightarrow M \times [0,1]$ is the inclusion at time $t$). 
\end{lemma}
\begin{proof}
 Let $\alpha \in H^*(M)$.  Let  $\alpha_t \in \mathcal{A}_t$ with $[\alpha_t]=\alpha$. Let   $\alpha_t'$ such that $ \frac{d}{dt} \alpha_t = d \alpha'_t $ and $\langle \alpha_t', \omega \rangle=0$ for each $\omega \in \mathcal{A}$. $\alpha_t'$ is determinate up to an exact form.  Set $ \tilde{\alpha} = \alpha_t + \alpha'_t dt \in \Omega^*(M \times [0,1])$. Pick a base of $H^*(M)$ and  let
$ \tilde{ \mathcal{A}} $  be the span of the $ \tilde{\alpha}$ when $\alpha$ run on the elements of this base.   
\end{proof}




\begin{lemma}  \label{parameter-propagator-lemma}
For $i=0,1$, let $P_i$ be a propagator compatible with $\mathcal{A}_i$ as in Lemma  \ref{construction-propogator} . There exists a differential form $\tilde{P} \in \Omega_2(\text{Conf}_2(L) \times [0,1])$ such that 
\begin{enumerate}
\item $i_i^*(\tilde{P})= P_i , i=0,1$
\item $d\tilde{P} = \sum_i (\tilde{\alpha}_i \times \tilde{\beta}_i  + \tilde{\beta}_i \times \tilde{\alpha}_i ) $
\item $  [i_{\partial}^*( P)] = \omega \otimes 1 \in H^*(\partial \text{Conf}_2(M) )   $,
where $\omega $ is the generator of $H^2(S^2)$ .
\item $ \text{sw}^*(P) = -P  $
\item $\int_{\text{Conf}_2(M)} \tilde{P} \wedge \pi_1^*(\tilde(\alpha)') \wedge \pi_2^*(\tilde{\alpha}'') =0 \in  \Omega^*( [0,1] ) $  for each   $ \tilde{\alpha}', \tilde{\alpha}'' \in \tilde{\mathcal{A}} .$
\end{enumerate}
\end{lemma}
In the last point the integral is considered  as the push forward of differential forms on $ \text{Conf}_2(M) \times [0,1] $ throughout the projection $ \text{Conf}_2(M) \times [0,1] \rightarrow [0,1]$.

The construction of $\tilde{P}$ follows the same lines of  the argument   in Lemma  \ref{construction-propogator}. Pick a closed two differential form $\tilde{\eta} \in \Omega^2(\partial \text{Conf}_2(M) \times [0,1] )$ such that $i_i^*(\tilde{\eta})= i_{\partial}^* (P_i) , i=0,1$, and extend it on a neighborhood of $\partial \text{Conf}_2(M) \times [0,1] $. We can then construct as before a $\tilde{P}' \in \Omega_2(\text{Conf}_2(L) \times [0,1])  $ which satisfies property $(2), (3)$ and $(4)$ of the lemma. 
From $(2)$ follows that $d_t \int_{\text{Conf}_2(M)} \tilde{P'} \wedge \pi_1^*(\tilde(\alpha)') \wedge \pi_2^*(\tilde{\alpha}'') =0 $  for each   $ \tilde{\alpha}', \tilde{\alpha}'' \in \tilde{\mathcal{A}} .$ Finally, we  add to $\tilde{P}'$ a suitable element of $ \pi_1^*(\tilde{\mathcal{A}}) \wedge \pi_2^*(\tilde{\mathcal{A}})$ in order to satisfy property $(5)$.




\begin{thebibliography}{10}











\bibitem{FO} K. Fukaya; K. Ono, \emph{Arnold conjecture and Gromov-Witten invariant},  Topology  38  (1999),  no. 5, 933--1048.


\bibitem{KS} M. Kontsevich and Y. Soibelman, \emph{Stability Structures, Motivic Donaldson-Thomas Invariants and Cluster Transformations} , arXiv:0811.2435


\bibitem{CS} V. Iacovino, \emph{Master Equation and Perturbative Chern-Simons theory}, arXiv:0811.2181.




\bibitem{OGW1} V. Iacovino, \emph{Open Gromov-Witten theory on Calabi-Yau three-fold I}, arXiv:0907.5225.

\bibitem{OGW2} V. Iacovino,  \emph{Open Gromov-Witten theory on Calabi-Yau three-folds II}, arXiv:0908.0393.


\bibitem{frame} V. Iacovino, \emph{Frame ambiguity in Open Gromov-Witten invariants}, arXiv:1003.4684.




\bibitem{OGW3} V. Iacovino,  \emph{Open Gromov-Witten theory without Obstruction}, arXiv:1711.05302.



\bibitem{KSWCF} V. Iacovino, \emph{Kontsevich-Soibelman Wall Crossing Formula and Holomorphic Disks}, arXiv:1711.05306.


\bibitem{Boundary-States} V. Iacovino,  \emph{Open Gromov-Witten invariants and Bounrdary States}, .












\bibitem{W} E. Witten, \emph{Chern-Simons Gauge Theory As A String Theory}, arXiv:hep-th/9207094.



\end{thebibliography}
\end{document}